\documentclass{amsart}
\usepackage{amsthm}
\usepackage{amsmath}
\usepackage{amsrefs}
\usepackage{amsfonts}
\usepackage{verbatim}
\usepackage{hyperref}
\usepackage{enumitem}
\usepackage{color}

\theoremstyle{plain}
\newtheorem{thm}{Theorem}

\theoremstyle{plain}
\newtheorem{lem}[thm]{Lemma}

\theoremstyle{definition}
\newtheorem{mydef}{Definition}

\theoremstyle{remark}

\newcommand{\abs}[1]{\ensuremath{\left|#1\right|}}
\newcommand{\ip}[2]{\ensuremath{\left\langle #1, #2 \right\rangle}}

\newcommand{\norm}[1]{\ensuremath{\left\| #1 \right\|}}

\begin{document}
\parindent 0pt     
\parskip 6pt
\numberwithin{equation}{section}

\title{A class of Markov chains with no spectral gap}
\author{Yevgeniy Kovchegov}
\address{Department of Mathematics, Oregon State University, Corvallis,
  OR 97331}
\email{kovchegy@math.oregonstate.edu}

\author{Nicholas Michalowski}
\address{Department of Mathematics, Oregon State
  University, Corvallis, OR 97331}
\email{Nicholas.Michalowski@math.oregonstate.edu}

\begin{abstract}
In this paper we extend the results of the research started in \cite{YK2009} and \cite{YK2010}, in which Karlin-McGregor diagonalization of certain reversible Markov chains over countably infinite general state spaces by orthogonal polynomials was used to estimate the rate of convergence to a stationary distribution.   

We use a method of Koornwinder \cite{TK1984} to  generate a large and interesting family of random walks which exhibits a lack of spectral gap, and a polynomial rate of convergence to the stationary distribution.  For the Chebyshev type subfamily of Markov chains, we use asymptotic techniques to obtain an upper bound of order $O\left({\log{t} \over \sqrt{t}}\right)$ and a lower bound of order $O\left({1 \over \sqrt{t}}\right)$ on the distance to the stationary distribution regardless of the initial state. Due to the lack of a spectral gap, these results lie outside the scope of geometric ergodicity theory \cite{SMRT2009}.

\end{abstract}

\maketitle

\section{Introduction}
Let $P=\Big(p(i,j)\Big)_{i,j \in \Omega}$ be a reversible Markov chain over a sample space $\Omega$, that is, it must satisfy the following {\it detailed balance conditions}:
$$\pi_i p(i,j)=\pi_jp(j,i) \qquad \forall i,j \in \Omega,$$
where $\pi$ is a non-trivial non-negative function over $\Omega$. If $P$ admits a unique stationary distribution $\nu$, then ${1 \over \sum\limits_{i \in \Omega} \pi_i}\pi=\nu$.

It can be shown that the reversible $P$ is a self-adjoint operator in $\ell^2(\pi)$, the space generated by the following inner product induced by $\pi$
 $$\ip{f}{g}_{\pi}=\sum_{i \in \Omega} f(i)g(i)\pi_i$$
 If $P$ is a tridiagonal operator (i.e.\ a nearest-neighbor random walk) on  $\Omega=\{0,1,2,\dots\}$, then it must have a simple spectrum, and is diagonalizable via orthogonal polynomials as it was studied in the 50's by Karlin and McGregor, see \cite{SKJM1959}, \cite{GS1975}, and \cite{MI2005}. There, the extended eigenfunctions $Q_j(\lambda)$ satisfying $Q_0 \equiv 1$ and $$P\begin{pmatrix}
  Q_0(\lambda)\\
  Q_1(\lambda)\\
  Q_2(\lambda)\\
  \vdots
\end{pmatrix} =
\lambda\begin{pmatrix}
  Q_0(\lambda)\\
  Q_1(\lambda)\\
  Q_2(\lambda)\\
  \vdots
\end{pmatrix}
$$
 are orthogonal polynomials with respect to a probability measure
 $\psi$.  If we let $p_t(i,j)$ denote the entries of the operator
 $P^t$ that represent $t$ step transition probabilities from state $i$ to state $j$ then
 $$p_t(i,j)=\pi_j \int_{-1}^1 \lambda^t Q_i(\lambda) Q_j(\lambda) d\psi(\lambda)~~~\forall i,j \in \Omega, $$
 where $\pi_j$ with $\pi_0=1$ is the reversibility measure of $P$.

 We will use the following distance to measure the deviation from the stationary distribution on a scale from zero to one. 
 \begin{mydef}
If $\mu$ and $\nu$ are two probability distributions over a sample space $\Omega$, then the {\it total variation distance} is 
$$\| \nu - \mu \|_{TV} = {1 \over 2} \sum_{x \in \Omega} |\nu(x)-\mu(x)|=\sup_{A \subset \Omega} |\nu(A)-\mu(A)|$$
\end{mydef}
 \noindent
 Let $\rho=\sum_{k=0}^{\infty} \pi_k$. Observe that $\rho < \infty$ if and only if the random walk $P$ is positive recurrent. Recall that $\nu={1 \over \rho}\pi$ is the stationary probability distribution. If in addition to being positive recurrent, the aperiodic nearest neighbor Markov chain originates at site $j$, then
 the total variation distance between the distribution $\mu_t=\mu_0P^t$ and $\nu$ is given by
 \begin{equation} \label{eqTV}
 \left\|\nu - \mu_t \right\|_{TV}  =  {1 \over 2} \sum_{n \in \Omega} \pi_n \left|\int_{(-1,1)} \lambda^t Q_j(\lambda) Q_n(\lambda) d\psi(\lambda)\right|,
 \end{equation} 
 as measure $\psi$ contains a point mass of weight ${1 \over \rho}$ at $1$. 
 See \cite{YK2009}.
 
 The rates of convergence are quantified via mixing times, which for an infinite state space with a unique stationary distribution are defined as follows. Here the notion of a mixing time depends on the state of origination $j$ of the Markov chain. See \cite{YK2010}.
\begin{mydef}
 Suppose $P$ is a Markov chain with a stationary probability distribution $\nu$ that commences at $X_0=j$. Given an $\epsilon >0$, the mixing time $t_{mix}(\epsilon)$ is defined as
 $$t_{mix}(\epsilon)=\min\left\{t~:~\|\nu-\mu_t\|_{TV} \leq \epsilon \right\}$$
\end{mydef}
\vskip 0.2 in
\noindent
 In the case of a nearest-neighbor process  on $\Omega=\{0,1,2,\dots\}$ commencing at $j$,  the corresponding mixing time has the following simple expression in orthogonal polynomials
$$ t_{mix}(\epsilon)=\min\left\{t~:~\sum_{n} \pi_n \left|\int_{(-1,1)} \lambda^t Q_j(\lambda) Q_n(\lambda) d\psi(\lambda)\right| \leq 2 \epsilon \right\},$$

Investigations into the use of  orthogonal polynomial techniques (see \cite{SKJM1959}, \cite{GS1975}) in the estimation of mixing times and distance to the stationary distribution has been carried out in \cite{YK2010} for certain classes of random walks.   In this paper we consider the problem from the other direction.  Namely given a large class of orthogonal polynomials we outline how to find the corresponding random walk and estimate the rate for the distance to the stationary distribution.  

More specifically beginning with the Jacobi polynomials, whose weight
function lies in $(-1,1)$ we use Koornwinder's techniques
\cite{TK1984} to attach a point mass at $1$.  For the class
of Jacobi type polynomials $Q_n$ thus obtained, the three term recurrence
relationship is understood \cite{HKJW1996}.  The tridiagonal operator
corresponding to these polynomials is not a Markov chain, however the
operator can be deformed to become one.  The corresponding changes in
the polynomials are easy to trace.  This gives a four parameter family
of nearest neighbor Markov chains whose distance to the stationary
distribution decays in a non-geometric way.  In principle the
asymptotic analysis presented in this paper can be applied to the
entire four parameter family.  We outline how this proceeds for
Chebyshev-type subfamily consisting of taking $\alpha=\beta=-1/2$ in
the Koornwinder class.

We would like to point out the important results of V.~B.~Uvarov
\cite{VU1969} on transformation of orthogonal polynomial systems by
attaching point masses to the orthogonality measure, predating the
Koornwinder results by fifteen years.  The results of V.~B.~Uvarov can
potentially be used in order to significantly extend the scope of
convergence rate problems covered in this current manuscript.

The paper is organized as follows. In Section \ref{sec:koorn} we discuss constructing positive recurrent Markov chains from the Jacobi family of orthogonal polynomials adjusted by using Koornwinder's techniques  to place a point mass at $x=1$. Next, we derive an asymptotic upper bound on the total variation distance to the stationary distribution in the case of general $\alpha>-1$ and $\beta>-1$ in Section \ref{sec:asympt}. Our main result, Theorem~\ref{main}, is presented in Section \ref{sec:chebyshev}. There, for the case of Chebyshev type polynomials corresponding to $\alpha=\beta=-1/2$, we produce both asymptotic lower and upper bounds for the total variation distance. Finally, in Section~\ref{comparison} we compare our main result to related results obtained by other techniques.

\section{From Orthogonal Polynomials to Random Walks via Koornwinder}\label{sec:koorn}

T.~Koornwinder \cite{TK1984} provides a method for finding the
orthogonal polynomials whose weight distribution is obtained from the
standard Jacobi weight functions
$C_{\alpha,\beta}(1-x)^{\alpha}(1+x)^{\beta}$ by attaching weighted
point masses at $-1$ and $1$. A spectral measure corresponding to a
Markov chain contains a point mass at $-1$ if and only if the Markov
chain is periodic. A spectral measure for an aperiodic Markov chain
contains a point mass at $1$ if and only if it is positive
recurrent. Thus in order to create a class of positive recurrent
aperiodic Markov chains with a Koornwinder type orthogonal polynomial
diagonalization we will only need to attach a point mass at $1$ and no
point mass at $-1$.

Let $N\geq 0$ and let $\alpha$, $\beta>-1$.  For $n=0, 1, 2, \ldots$
define
\begin{equation}\label{eq:Koornwinder}
P_n^{\alpha,\beta,N}(x)=\Big(\frac{(\alpha+\beta+2)_{n-1}}{n!}\Big)A_n\Big[-N
  (1+x)\frac{d}{dx}+B_n\Big]P_n^{\alpha,\beta}(x),
\end{equation}
where 
$$A_n=\frac{(\alpha+1)_n}{(\beta+1)_n},$$
$$B_n=\frac{(\beta+1)_nn!}{(\alpha+1)_n(\alpha+\beta+2)_{n-1}}+\frac{n(n+\alpha+\beta+1)N}{(\alpha+1)},$$
$P_n^{\alpha,\beta}$ is the standard Jacobi polynomials of degree $n$
and order $(\alpha,\beta)$, $(x)_n=x(x+1)\cdots(x+n-1)$.  These
polynomials form a system of orthogonal polynomials with respect to
the probability measure
$d\psi(x)={C_{\alpha,\beta}(1-x)^\alpha(1+x)^\beta dx+N\delta_1(x)
  \over N+1}$, where $C_{\alpha,\beta}={1 \over
  \mathcal{B}(\alpha+1,\beta+1)}$, \quad $\mathcal{B}(\cdot,\cdot)$ is
the beta function, and $\delta_1(x)$ denotes the a unit point mass
measure at $x=1$. See T.~Koornwinder \cite{TK1984}. Direct calculation
shows that $P_n^{\alpha,\beta,N}(1)=\frac{(\alpha+1)_n}{n!}$, and so
we normalize $Q_n(x)=n!P_n^{\alpha,\beta,N}(x)/(\alpha+1)_n$ which is
the orthogonal set of polynomials with respect to $d\psi$ satisfying
$Q_n(1)=1$.

As we have mentioned earlier, the tridiagonal operator $H$ corresponding to the recurrence relation of the orthogonal polynomials may not be a Markov chain operator.  Let $p_i$, $r_i$ and $q_i$ denote the coefficients in the tridiagonal recursion
$$p_i Q_{i+1}(x)+r_iQ_i(x)+q_i Q_{i-1}(x)=xQ_i(x),$$
for $i=0,1,2,\dots$, where we let $Q_{-1} \equiv 0$ as always.

Notice because the polynomials are normalized so that
$Q_i(1)=1$ it follows immediately that $p_i+r_i+q_i=1$.
However some of the coefficients  $p_i$, $r_i$, or $q_i$ may turn out to be negative, in which case the rows of the tridiagonal operator $A$ would add up to one, but will not necessarily consist of all nonnegative entries. 

In the case when all the negative entries are located on the main diagonal, this may be overcome by
considering the operator $\frac{1}{\lambda+1}(H+\lambda I)$.  For $\lambda \geq-\inf\limits_i r_i$ this ensures all entries in the matrix $\frac{1}{\lambda+1}(H+\lambda I)$ are nonnegative and hence can be thought of as transition probabilities.  More generally, if a polynomial $p(\cdot)$ with coefficients adding up to one is found to satisfy $p(H) \geq 0$ coordinatewise, then such $p(H)$ would be a Markov chain.

\section{An Asymptotic Upper Bound for Jacobi type Polynomials}\label{sec:asympt}

In this section we derive asymptotic estimates for the distance to the
stationary distribution when our operator given by
$P_{\lambda}=\frac{1}{\lambda+1}(H+\lambda I)$ is a Markov chain. In
this case the Karlin-McGregor orthogonal polynomials for $P_{\lambda}$
are $Q_j\Big((1+\lambda)x-\lambda\Big)$ and the orthogonality
probability measure is 
${1 \over 1+\lambda}d\psi\Big((1+\lambda)x-\lambda\Big)$ over 
$\Big({\lambda-1 \over \lambda+1}, 1 \Big]$, where the $Q_j$ are the
Jacobi type polynomials introduced by Koornwinder from the previous
section.

Of course the new operator $P_{\lambda}$ is again tridiagonal.  For the $n$-th row of $P_{\lambda}$, let us denote the $(n-1)$-st, $n$-th, and $(n+1)$-st entries by $q_n^\lambda$, $r_n^\lambda$, and $p_n^\lambda$ respectively. Here the entries of $P_{\lambda}$ can be expressed via the entries of $H$ as follows
$$p_n^\lambda=\frac{p_n}{1+\lambda}, \qquad r_n^\lambda=\frac{r_n+\lambda}{1+\lambda}, \quad \text{ and } \quad q_n^\lambda=\frac{q_n}{1+\lambda}$$  
Clearly we still have that
$p_n^\lambda+r_n^\lambda+q_n^\lambda=1$.  

With the probabilities in hand we now compute the corresponding reversibility function $\pi_n^\lambda$ of $P_{\lambda}$ which is equal to the corresponding function of $H$ defined as $\pi_n=\frac{p_0\cdots p_{n-1}}{q_1\cdots q_n}$.  Here $\pi_0^\lambda=1=\pi_0$ and $\pi_n^\lambda=\frac{p_0^\lambda\cdots p_{n-1}^\lambda}{q_1^\lambda\cdots q_n^\lambda}=\frac{p_0\cdots p_{n-1}}{q_1\cdots q_n}=\pi_n$.

Changing variables in \eqref{eqTV} yields
$$\|\nu-\mu_t\|_{TV}=\frac{1}{2}\sum_{n=0}^{\infty}\pi_n\bigg|\int_{(-1,1)}\Big(\frac{x}{1+\lambda}+\frac{\lambda}{1+\lambda}\Big)^tQ_j(x)Q_n(x)\,d\psi(x)\bigg|$$

\begin{lem}\label{lem:bounds}
Consider the case when $p_n>0$ and $q_n>0$ for all $n \geq 0$, and \mbox{$\infty> \lambda \geq-\inf\limits_i r_i$}.
Then, for the Jacobi type polynomials $Q_j$ the distance to the stationary distribution satisfies the following bound
\begin{equation}\label{jacobi_bound}
\norm{\nu-\mu_t}_{TV}\leq\frac{C_{\alpha,\beta,\lambda}\norm{Q_j}_\infty}{(t+1)^{1+\alpha}}\sum_{n=0}^{t+j}\pi_n\norm{Q_n}_\infty+\frac{1}{2}\sum_{n=j+t+1}^{\infty}\pi_n
\end{equation}
for a certain constant $C_{\alpha,\beta,\lambda}$.
\end{lem}

\begin{proof}
For $n>j+t$, it follows from the orthogonality of the polynomials and our normalization $Q_i(1)=1$ that
$$\int_{(-1,1)}\Big(\frac{x}{1+\lambda}+\frac{\lambda}{1+\lambda}\Big)^tQ_j(x)Q_n(x)\,d\psi(x)=1$$
 It is then easy to see that 
$\norm{\nu-\mu}_{TV}\leq I+II+\frac{1}{2}\sum_{n=j+t+1}^{\infty}\pi_n$, where 
$$I=\frac{1}{2}\sum_{n=0}^{j+t}\pi_n\int_{(-1,0)}\abs{\Big(\frac{x+\lambda}{1+\lambda}\Big)^tQ_j(x)Q_n(x)}(1-x)^\alpha(1+x)^\beta\,dx$$
$$\text{ and } \quad II=\frac{1}{2}\sum_{n=0}^{j+t}\pi_n\int_{(0,1)}\Big(\frac{x+\lambda}{1+\lambda}\Big)^t\big|Q_j(x)Q_n(x)\big|(1-x)^\alpha(1+x)^\beta\,dx$$

To estimate $I$ notice that $\abs{\frac{x+\lambda}{1+\lambda}}\leq
\max(\frac{\lambda}{1+\lambda}, \abs{\frac{1-\lambda}{1+\lambda}})<1$ for
$\lambda>0$. Hence \mbox{$I\leq A_j(|t|)e^{-ct}$} for an appropriate polynomial $A_j(\cdot)$ such that 
$$\frac{1}{2}\|Q_j\|_{\infty}\sum_{n=0}^{j+t}\pi_n \|Q_n\|_{\infty}\int_{(-1,0)}(1-x)^\alpha(1+x)^\beta\,dx \leq A_j(|t|),$$
and 
$c=-\log\left\{\max(\frac{\lambda}{1+\lambda}, \abs{\frac{1-\lambda}{1+\lambda}})\right\}$. Such polynomial $A_j$ exists since $\|Q_n\|_{\infty}$ grows polynomially in $n$ and $\pi_n$ is bounded. See formula 22.14.1 in Abramowitz and Stegun \cite{MAIS1972}.

Thus $I$ is clearly bounded by the right hand side of \eqref{jacobi_bound}.

For the second term,
$II\leq \frac{1}{2}\sum_{n=0}^{j+t}\pi_n\norm{Q_nQ_j}_{\infty}\int_0^1\Big(\frac{x+\lambda}{1+\lambda}\Big)^t\big(1-x)^\alpha(1+x)^\beta\,dx$.
There we make the change of variables $s=-\log(\frac{x+\lambda}{1+\lambda})$, and for
simplicity let $x(s)=(1+\lambda)e^{-s}-\lambda$.  Then the
integral reduces to
$$(1+\lambda)^{1+\alpha}\int_0^{\log(\frac{1+\lambda}{\lambda})}e^{-s(t+1)}\big(1-\lambda+(1+\lambda)e^{-s})^\beta\big(1-e^{-s}\big)^\alpha\,ds$$

Using the fact that $(1-e^{-s})^\alpha=s^\alpha\Big(1+O(s)\Big)$ and 
$\big(1-\lambda+(1+\lambda)e^{-s})^\beta=2^\beta+O(s)$,
the above integral becomes
$$(1+\lambda)^{1+\alpha}\int_0^{\log(\frac{1+\lambda}{\lambda})}e^{-s(t+1)}\Big(2^\beta
s^\alpha+O(s^{\alpha+1})\Big)\,ds,$$
where the upper bounds $O(s)$ can be made specific.
Next, applying the standard asymptotic methods of Laplace to this yields the following asymptotics  
$$\int_0^{\log(\frac{1+\lambda}{\lambda})}e^{-s(t+1)}
s^\alpha\,ds ~\asymp~ \frac{\Gamma(\alpha+1)}{(t+1)^{1+\alpha}}$$

Thus one can obtain a large enough constant $\widetilde{C}_{\alpha,\beta,\lambda}$ such that
$$II\leq\
\frac{\widetilde{C}_{\alpha,\beta,\lambda}\norm{Q_j}_\infty}{(t+1)^{1+\alpha}}\sum_{n=0}^{t+j}\pi_n\norm{Q_n}_\infty$$  
\end{proof}

In order to derived effective bounds on $\norm{\nu-\mu_t}_{TV}$ it is
necessary to gain a more detailed understanding of $\pi_n$
and $\norm{Q_n}_\infty$.  When $\min(\alpha,\beta)\geq -\frac{1}{2}$, the $\norm{Q_n}_\infty$ can be
estimated using the known maximum for the Jacobi polynomials found in Lemma 4.2.1 on page 85 of \cite{MI2005} together
with Koornwinder's definition of these polynomials.

One way to derive estimates for $\pi_n$ is to use the expression
$\pi_n$ in terms of $p_n$,  $r_n$, and $q_n$.  For Koorwinder's class of
polynomials these expressions are derived for all $\alpha, \beta, M,
N$ in \cite{HKJW1996}.  It can be verified directly that in the case
when $M=0$, then
$p_0=\frac{2(\alpha+1)}{(1+N)(\alpha+\beta+2)}>0$. After taking into
account the normilization $Q_n(1)=1$, and taking into account a
small typo, it can be verified from equations (41)--(45) in \cite{HKJW1996} that $p_n$ and
$q_n$ are positive for $n\geq 1$.  Thus the conditions for Lemma~\ref{lem:bounds} are satisfied for all $\alpha, \beta >-1$. Furthermore, from (18), (19) and (32)  in \cite{HKJW1996}
it can be easily seen that $p_n\to \frac{1}{2}$ and $q_n\to \frac{1}{2}$
as $n\to\infty$, and hence $r_n=1-p_n-q_n\to 0$ as $n\to \infty$.  Thus
for $\lambda$ large enough the operator $P^\lambda$ corresponds to a
Markov chain.

As the expressions for these quantities
laborious to write down, instead we focus our attention on a specific
case in which our calculations are easy to follow.  Specifically
we focus on the Chebyshev polynomials.

\section{Chebyshev Polynomials: Upper and Lower Bounds}\label{sec:chebyshev}
By applying Koorwinder's results to the Chebyshev polynomials of the first kind which correspond to the case of $\alpha=\beta=-{1 \over 2}$, we arrive at a family of orthogonal polynomials with respect to the measure
$\frac{1}{1+N}\Big(\frac{1}{\pi\sqrt{1-x^2}}dx+N\delta_1(x)\Big)$.
Using \eqref{eq:Koornwinder} we find that here,
$$Q_n(x):=-N(x+1)U_{n-1}(x)+(1+2nN)T_n(x),$$ 
where $T_n$ and $U_n$ denote the Chebyshev polynomials of the first and second kind
respectively.  Notice that $U_n(1)=n+1$ and $T_n(1)=1$, which immediately to verify that $Q_n(1)=1$.

Once again we consider the operator 
$$H=\begin{pmatrix}r_0 & p_0 & 0 & 0 & 0 &\cdots\\
q_1 & r_1 & p_1 & 0 & 0 & \vdots\\
0 & q_2 & r_2 & p_2 & 0 & \ddots\\
0  & 0 & q_3 & r_3 & \ddots & \ddots\\
\vdots & \cdots & \ddots & \ddots & \ddots & \dots\\
\end{pmatrix},$$
on $\ell^2(\pi)$, so that vector $(Q_0(x), Q_1(x), Q_2(x),
\ldots)^T$ is an eigenvector with eigenvalue~$x$.

Specifically the numbers $p_n$,  $r_n$, and $q_n$ satisfy
$p_0P_1(x)+r_0P_0(x)=x$ for $n=0$, and 
\begin{equation}\label{eigvaleq}
  p_nQ_{n+1}(x)+r_nQ_n(x)+q_nQ_{n-1}(x)=xQ_n(x)\qquad \text{for $n\geq 1$.}
\end{equation}

Keisel and Wimp \cite{HKJW1996} give expressions
for $p_n$,  $r_n$ and $q_n$ for $n\geq 0$.  To find the expressions
directly in this case one could use \eqref{eigvaleq} to derive three
linearly independent equations, and solve for $p_n$,  $r_n$, and $q_n$.

For the case $n=0$ the equation immediately gives us that
$p_0=\frac{1}{N+1}$ and $r_0=\frac{N}{N+1}$.  Evaluating at convenient
choices of $x$, such as $-1, 0, 1$, do not yield linearly
independent equations for all $n$.  One solution to this is to
evaluate at $x=1,-1$ and differentiate \eqref{eigvaleq} and then
evaluate at $x=0$.  This gives three linearly independent equations
and a direct calculation then shows that
\begin{equation}
\begin{gathered}\label{eq:linearprobs}
p_n=\frac{1}{2}\cdot\frac{1+(2n-1)N}{1+(2n+1)N}, \quad
q_n=\frac{1}{2}\cdot\frac{1+(2n+1)N}{1+(2n-1)N}, \quad \text{and}\\
r_n=\frac{-2N^2}{(1+(2n-1)N)(1+(2n+1)N)}
\end{gathered}
\end{equation}
As $r_n\leq 0$ the operator $H$ fails to correspond to a Markov chain.
However this is the case we addressed at the end of Section~\ref{sec:koorn} of the current paper.
Thus consider $P_\lambda=\frac{1}{1+\lambda}(H+\lambda I)$.  Now, since $|r_n|$ is
a decreasing sequence for $n\geq 1$. So provided that $\lambda\geq |r_1|=\frac{2N^2}{(1+N)(1+3N)}$, we then have 
$p_n^\lambda, r_n^\lambda, q_n^\lambda\geq 0$.  Thus we can consider these coefficients  $p_n^\lambda$,  $r_n^\lambda$, and $q_n^\lambda$ as the transition probabilities in a nearest neighbor random walk.

Recall that $\pi_n^\lambda=\pi_n=\frac{p_0\cdots p_{n-1}}{q_1\cdots
  q_n}$.  Thus for $P_\lambda$ we can directly calculate $\pi_n$ from
\eqref{eq:linearprobs}.  We have that $p_0\cdots
p_{n-1}=\frac{1}{2^{n-1}}\frac{N}{1+(2n-1)N}$ and  similarly $q_1\cdots q_n=\frac{1}{2^n}\frac{1+(2n+1)N}{1+N}$.  Thus $\pi_n=\frac{2(1+N)N}{(1+(2n-1)N)(1+(2n+1)N)}$.

\begin{thm}\label{main} Given $N>0$ and $\lambda\geq \frac{2N^2}{(1+N)(1+3N)}$.
  Consider the case of the Chebyshev-type random walks over $\Omega=\{0,1,2,\dots \}$ with probability operator
  $$P_\lambda=\begin{pmatrix}r_0^\lambda & p_0^\lambda & 0 & 0 & 0 &\cdots\\
q_1^\lambda & r_1^\lambda & p_1^\lambda & 0 & 0 & \vdots\\
0 & q_2^\lambda & r_2^\lambda & p_2^\lambda & 0 & \ddots\\
0  & 0 & q_3^\lambda & r_3^\lambda & \ddots & \ddots\\
\vdots & \cdots & \ddots & \ddots & \ddots & \dots\\
\end{pmatrix},$$
  where $p_n^\lambda=\frac{1}{2(1+\lambda)}\cdot\frac{1+(2n-1)N}{1+(2n+1)N}$, \quad $q_n^\lambda=\frac{1}{2(1+\lambda)}\cdot\frac{1+(2n+1)N}{1+(2n-1)N}$
  and \quad $r_n^\lambda=1-p_n^\lambda-q_n^\lambda$ for $n \geq 1$, with $p_0^\lambda=\frac{1}{(1+\lambda)(N+1)}=1-r_0^\lambda$.  
  
  Then for the random walk originating at some site $j \in \Omega$, there are positive constants $c$ and $C$ that depend on $j$, $N$ and $\lambda$ such that 
  $$\frac{c}{\sqrt{t}}\leq \norm{\nu-\mu_t}_{TV}\leq
  C\frac{\log t}{\sqrt{t}}$$
  for $t$ sufficiently large.
  
\end{thm}
\begin{proof}
  For the upper bound we simply need to estimate the sums appearing in Lemma~\ref{lem:bounds}.  Since $\pi_n=O\big(\frac{1}{(n+1)^2}\big)$, it is easy to see that the second sum $\sum_{n=j+t+1}^{\infty}\pi_n$ is bounded by $C_N/(t+j+1)$.  
  The main term turns out to be the first sum.

  In the case of the Chebyshev type polynomials we have the bound
  \mbox{$\norm{Q_n}_\infty\leq 4Nn+1$}.  Thus the first sum in Lemma~\ref{lem:bounds} is bounded by 
  $\hat C_{\alpha,\beta,\lambda,N}\frac{j\log(t+j+2)}{\sqrt t}$ for an appropriate constant  $\hat C_{\alpha,\beta,\lambda,N}$.
 And so, for an appropriate $C$ and large $t$,
  $$\norm{\nu-\mu_n}_{TV}\leq C\frac{\log t}{\sqrt{t}}$$

  On the other hand, recalling that $Q_0(x)=\pi_0=1$, we have that:
  $$\norm{\nu-\mu_n}_{TV}\geq \abs{\int_{(-1,1)}\Big(\frac{x+\lambda}{1+\lambda}\Big)^tQ_j(x)(1+x^\beta)(1-x)^\alpha\,dx}$$
 However we have already shown that for large enough $t$, the above right-hand side is asymptotic to
  $\frac{\tilde C}{\sqrt{1+t}}$.
\end{proof}

We finish with some concluding remarks.  At first the bound
\mbox{$\norm{Q_n}_\infty\leq 4Nn+1$} may appear somewhat imprecise since near $x=1$, we have that $Q_n(1)=1$.  
It is tempting to suggest that the correct asymptotic for the total
variation norm is $C/\sqrt{t}$.  However on closer examination in the
neighborhood of $x=1$, $Q_n'(x)\approx n^3$.  This $n^3$ causes the
errors to be at least of the order of  the main term.  Overall
it seems unlikely to the authors that $C/\sqrt{t}$ is the correct asymptotic for the Chebyshev-type polynomials.

\section{Comparison to other methods}\label{comparison}

An ergodic Markov chain $P=\Big(p(i,j)\Big)_{i,j \in \Omega}$ with stationary distribution $\nu$ is said to be {\it geometrically ergodic} if and only if there exists $0<R<1$ and a function $M: \Omega \rightarrow \mathbb{R}_+$ such that for each initial state $i \in\Omega$, the total variation distance 
decreases exponentially as follows
$$\|p_t(i,\cdot)-\nu(\cdot)\|_{TV}={1 \over 2} \sum\limits_{j \in \Omega} |p_t(i,j)-\nu(j)| \leq M(i) R^t$$
In other words, an ergodic Markov chain is geometric when the rate of convergence to stationary distribution is exponential. See \cite{SMRT2009} and references therein.

If the state space $\Omega$ is finite, $|\Omega|=d <\infty$, and Markov chain is irreducible and aperiodic, then $P$ will have eigenvalues that can be ordered as follows $$\lambda_1=1 >|\lambda_2| \geq \dots \geq |\lambda_d|$$
In which case, the Perron-Frobenious Theorem will imply geometric ergodicity with 
$$\|p_t(i,\cdot)-\nu(\cdot)\|_{TV}=O(t^{m_2-1} |\lambda_2|^t),$$
where $m_2$ is the algebraic multiplicity of $\lambda_2$. Here the existence of a positive {\it spectral gap}, $1-|\lambda_2|>0$, implies geometric ergodicity with the exponent $-\log  |\lambda_2| \approx 1- |\lambda_2|$ whenever the spectral gap is small enough.

When dealing with Markov chains over general countably infinite state space $\Omega$, the existence of a positive spectral gap of the operator $P$ is essentially equivalent  to the chain being geometrically ergodic. For instance, the orthogonal polynomial approach in \cite{YK2010} resulted in establishing the geometric rate $R=\max\left\{r+2\sqrt{pq},~{q \over q+r}\right\}$ for the Markov chain
$$P=\left(\begin{array}{ccccc}0 & 1 & 0 & 0 & \dots \\q & r & p & 0 & \dots \\0 & q & r & p & \ddots \\0 & 0 & q & r & \ddots \\\vdots & \vdots & \ddots & \ddots & \ddots\end{array}\right) \qquad q>p,~~~r>0$$ 
over $\Omega=\mathbb{Z}_+$, together with establishing the value of the spectral gap, $1-r>0$.

As for the Markov chain $P_\lambda$ considered in Theorem~\ref{main} of this paper, its spectral measure ${1 \over 1+\lambda}d\psi\Big((1+\lambda)x-\lambda\Big)$ over $\Big({\lambda-1 \over \lambda+1}, 1 \Big]$ admits {\it no} spectral gap between the point mass at $1$ and the rest of the spectrum implying sub-geometric ergodicity. The sub-exponential rate in total variation norm is then estimated to be of polynomial order between $\frac{1}{\sqrt{t}}$ and $\frac{\log t}{\sqrt{t}}$.

In the field of probability and stochastic processes, there is a great
interest in finding methods for analyzing Markov chains over general
state space that have polynomial rates of convergence to stationary
distribution. In Menshikov and Popov \cite{MMSP1995} a one dimensional
version of Lamperti's problem is considered. There, a class of ergodic
Markov chains on countably infinite state space with sub-exponential
convergence to the stationary probabilities is studied via
probabilistic techniques. One of their results relates to our main
result, Theorem~\ref{main}.  Namely, Theorem 3.1 of \cite{MMSP1995}
when applied to our case, implies for any $\varepsilon>0$ the
existence of positive real constants $C_1$ and $C_2$ such that
$$C_1t^{-{1 \over 2}-\varepsilon} \leq |\nu(0)-\mu_t(0)| \leq C_2t^{-{1 \over 2}+\varepsilon}$$
Thus for the Markov chain considered in Theorem~\ref{main}, the
orthogonal polynomials approach provides a closed form expression for
the difference $\nu-\mu_t$, and a significantly sharper estimate on
convergence of $\mu_t$ to the stationary distribution $\nu$, for both
the single state distance $|\nu(0)-\mu_t(0)|$ and a much stronger
total variation norm, $\|\nu-\mu_t\|_{TV}$.

\section{Acknowledgments}
We would like to thank Yuan Xu of University of Oregon for his helpful comments that
initiated this work. We would also like to thank Michael Anshelevich of Texas A \& M for the feedback he provided during the conference on orthogonal polynomials in probability theory in July of 2010. We would like thank Andrew R. Wade of the University of Strathclyde for his helpful comments on the preprint of this paper. Finally, we would like to thank the anonymous  referee for the many helpful corrections and suggestions.


\begin{bibdiv}
  \begin{biblist}
   \bib{MAIS1972}{book}{
        editor={M. Abramowitz},
        editor={I. A. Stegun},
        title={Handbook of Mathematical Functions with Formulas, Graphs, and Mathematical Tables},
        edition={Ninth}, 
        publisher={Dover Publications},  
        year={1972}
      }  
   \bib{MI2005}{book}{
        author={M. Ismail},
        title={Classical and quantum orthogonal polynomials in one variable},
        series={Encyclopedia of Mathematics and its Applications}, 
        number={98},
        publisher={Cambridge University Press},  
        year={2005}
      } 
   \bib{SKJM1959}{article}{
      title={Random Walks},
      author={S. Karlin},
      author={J. L. McGregor},
      journal={Illinois Journal of Math.},
      volume={3},
      date={1959},
      number={1},
      pages={417--431}
    }
    \bib{HKJW1996}{article}{
      title={A note on Koornwinder's polynomials with weight function $(1-x)^\alpha(1+x)^\beta+M\delta(x+1)+N\delta(x-1)$},
      author={Kiesel, K.},
      author={Wimp, J.},
      journal={Numerical Algorithms},
      volume={11},
      date={1996},
      pages={229--241}
    }
    \bib{TK1984}{article}{
      title={Orthogonal polynomials with weight function $(1-x)^{\alpha }(1+x)^{\beta }+M\delta (x+1)+N\delta (x-1)$},
      author={Koornwinder, T.}
      journal={Canad. Math. Bull.},
      volume={27},
      date={1984},
      number={2},
      pages={205--214}
    }
    \bib{YK2009}{article}{
      title={Orthogonality and probability: beyond nearest neighbor transitions},
      author={Kovchegov, Y.},
      journal={Electron. Commun. Probab.},
      volume={14},
      year={2009},
      pages={90--103}
    }    
    \bib{YK2010}{article}{
      title={Orthogonality and probability: mixing times},
      author={Kovchegov, Y.},
      journal={Electron. Commun. Probab.},
      volume={15},
      year={2010},
      pages={59--67}
    }
    \bib{SMRT2009}{book}{
        author={S. Meyn},
        author={R. L. Tweedie},
        title={Markov Chains and Stochastic Stability},
        edition={Second}, 
        publisher={Cambridge University Press},  
        year={2009}
      }  
    \bib{MMSP1995}{article}{
      title={Exact Power Estimates For Countable Markov Chains},
      author={Menshikov, M.~V.},
      author={Popov, S.~Yu.},
      journal={Markov Processes Relat. Fields},
      volume={1},
      date={1995},
      pages={57--78}
    }
      \bib{GS1975}{book}{
        author={G. Szeg\"{o}},
        title={Orthogonal Polynomials},
        edition={Fourth}, 
        publisher={AMS Colloquium Publications},  
        volume={23},
        year={1975}
      }  
    \bib{VU1969}{article}{
      title={Relation between systems of polynomials orthogonal with respect to various distribution functions},
      author={Uvarov, V.~B.},
      journal={Vychisl. Mat. i Mat. Fiz  (USSR)},
      volume={9},
      number={6},
      year={1969},
      pages={1253--1262}
    }
\end{biblist}
\end{bibdiv}
\end{document}